\documentclass[12pt]{amsart}

\usepackage{amsthm}
\usepackage{amssymb}
\usepackage{amsmath}
\usepackage{hyperref}
\usepackage{xcolor}
\usepackage{amsthm}
\usepackage{mathrsfs}
\usepackage{dsfont}
\usepackage[matrix,arrow]{xy}

\usepackage{enumerate}

\usepackage[style=trad-abbrv,maxnames=99,maxalphanames=9, isbn=false, giveninits=true, doi=true, url=false]{biblatex}
\bibliography{gradientestimate.bib}
\renewbibmacro{in:}{}

\DeclareMathAlphabet\mathbfcal{OMS}{cmsy}{b}{n}

\numberwithin{equation}{section}

\newtheorem{theorem}{Theorem}[section]

\newtheorem{lemma}[theorem]{Lemma}

\newtheorem{corollary}[theorem]{Corollary}
\newtheorem{remark}[theorem]{Remark}

\newtheoremstyle{named}{}{}{\itshape}{}{\bfseries}{.}{.5em}{\thmnote{#3\ }#1}
\theoremstyle{named}

 \newcommand{\RR}{\mathbb{R}}
 \newcommand{\CC}{{\mathbb C}}

\newcommand{\cE}{\mathcal{E}}

\newcommand{\cC}{\mathcal{C}}

\newcommand{\cL}{\mathcal{L}}

\newcommand{\cN}{\mathcal{N}}

\newcommand{\ddc}{\sqrt{-1}\partial\bar\partial}
\newcommand{\PSH}{\mathcal{PSH}}

\newcommand{\tr}{\mathrm{tr}}

\title{Complex Alexandrov-Bakelman-Pucci estimate and its applications
}

\usepackage{hyperref}
\hypersetup{colorlinks,linkcolor={blue},citecolor={red}}

\begin{document}

\author{Junbang Liu}
\begin{abstract}
    We prove an Alexandrov-Bakelman-Pucci type estimate, which involves the integral of the determinant of the complex Hessian over a certain subset. It improves the classical ABP estimate adapted  
    (by inequality $2^{2n}|\det(u_{i\bar{j}})|^2\geq |\det(\nabla^2u)|$) to complex setting. We give an application of it to derive sharp gradient estimates for complex Monge-Amp\`ere equations. The approach is based on the De Giorgi iteration method developed by Guo-Phong-Tong for equations of complex Monge-Amp\`ere type. 
\end{abstract}
\maketitle

\tableofcontents
\section{Introduction}
The classical Alexandrov-Bakelman-Pucci estimate (ABP-estimate in short) states that for a function $u\in C^2(\Omega)\cap C^0(\bar{\Omega})$ on a bounded domain $\Omega\subset \mathbb{R}^n$, we have \[
	\sup_{\Omega}u\leq \sup_{\partial\Omega}u+C_ndiam(\Omega)\left(\int_{\{-u=\Gamma_{-u}\}}|\det\nabla^2u|\right)^{\frac{1}{n}}.
\]
As a corollary, if the function $u$ satisfies the differential inequality $a^{ij}u_{ij}\geq f$ for some semipositive symmetric matrix-value function $a^{ij}$, then the integral on the right-hand side can be replaced by the $L^n$-norm of $\frac{f^{-}}{\mathscr{D^*}}$, where $\mathscr{D}^*=(\det(a^{ij}))^{\frac{1}{n}}$.
It's a fundamental tool in PDE theory, especially in the study of viscosity solutions for fully nonlinear equations. Both the result itself and its proof have numerical significant applications. In this paper, we prove a complex version of the ABP estimates. We find that in the complex setting, the integral of $|\det(\nabla^2 u)|$ over the contact set can be replaced by the integral of the determinant of complex Hessian $\det(u_{i\bar{j}})$ over certain subsets. Let $B_1\subset \CC^n$ be the unit ball and $u$ be a real function defined on $B_1$.
\begin{theorem}\label{refinedABP}
    Let $u\in C^2(B_1)\cap C^0(\overline{B_1})$, and let $\Phi:\RR\rightarrow \RR^+$ be a positive increasing function with $\int_0^\infty\Phi^{-\frac{1}{n}}(t)dt\leq \Lambda<\infty$. Then there exists constant $c_1,c_2$ depending only on $n,\Lambda$ such that\begin{equation}
        \sup_{B_1}u\leq \sup_{\partial B_1}u+c_1+c_2\left(\int_{\Gamma^-}\det(-u_{i\bar{j}})\Phi(\log(\det(-u_{i\bar{j}})))\right)^{\frac{1}{n}}.
    \end{equation}
    Here $\Gamma^-:=\{x\in B_1|-\ddc u(x)>0,u(x)>\sup_{\partial B_1}u\}.$ Moreover the constant $c_1$ can be chosen to be \[
    c_1=\min\{c(n),\left(\int_{\Gamma^-}\det(-u_{i\bar{j}})\right)^{\delta_n}\}, \text{ for dimensional constant } c(n),\delta_n.
    \]
\end{theorem}

To the author’s knowledge, the application of the ABP estimate to complex geometry can be traced back to Cheng and Yau, who suggested that the $C^0$-estimate for the complex Monge-Ampère equation could be proved using the ABP estimate (see Blocki’s paper \cite{MR2817572}). Székelyhidi employed this method to establish $C^0$-estimates for a class of fully nonlinear elliptic equations on compact Hermitian manifolds, assuming the existence of a $\cC$-subsolution \cite{MR3807322}. A major breakthrough was achieved by X. Chen and J. Cheng in \cite{MR4301557, MR4301558}, where they used the ABP estimate to derive uniform estimates for solutions to cscK-type equations, assuming an entropy bound and a uniform bound on the scalar curvature. Building on the work of Chen and Cheng, Guo, Phong, and Tong developed a purely PDE-based method in \cite{MR4593734} to obtain $L^\infty$-estimates for a class of Monge-Ampère-type equations, relying only on a weak Orlicz norm $(L^1 \log^p L$-norm, $p > n$). An important intermediate step in their approach involved using a technique from Chen and Cheng to apply the ABP estimate in order to derive exponential integrability for the solution (energy estimate). For applications to parabolic equations, see also \cite{MR4311100, MR4703037, zhao2023linfty}. All of the above applications of classical ABP-estimate to complex geometry are based on the observation $2^{2n}|\det(u_{i\bar{j}})|^2\geq |\det(\nabla^2 u)|$. So, our theorem can be viewed as an improvement of the classical one in a complex setting, and it leads to sharp estimates in practice. There have also been some improvements to the ABP estimate in the complex setting (see \cite{MR4366024, MR2968239} and references therein). We remark that when the function $\Phi$ grows slowly, which does not satisfy the condition $\int_0^\infty\Phi^{-\frac{1}{n}}<\infty$, we don't expect to bound the maximum uniformly. But in this case, we still have some Trudinger-type integral estimates, \begin{corollary}
    Fix $0<p<n$, let $u\in C^2(B_1)\cap C^0(\overline{B_1})$. Assume that $u|_{\partial B_1}\leq 0$. Then there exists a uniform constant $c_1,c_2,c_3$ depending on $n,p$ such that \[
    \int_{B_1}\exp\left\{c_1\left(\frac{(u-c_2)_+}{\cN_p^\frac{1}{n}}\right)^{\frac{n}{n-p}}\right\}dV\leq c_3,
    \]where $\cN_p:=\int_{\Gamma^-}\det(-u_{i\bar{j}})(\log^2(\det(-u_{i\bar{j}}))+1)^\frac{p}{2}dV.$
\end{corollary}

One more corollary of the refined ABP estimate is an estimate of the sup-slope of a class of Hessian-type equations. Let $(M^n,\omega_M)$ be a compact Hermitian manifold with a fixed Hermitian metric $\omega_M.$ We assume that $\Gamma$ is an open convex symmetric cone in $\RR^n$ such that $\Gamma_n:=\{\lambda_i>0, i=1,\ldots,n\}\subset \Gamma\subset \Gamma_1:=\{\sum_{i=1}^n\lambda_i>0\}$. Let $f$ be a smooth nonnegative  function defined on the closure of $\Gamma$ which satisfies the following structural conditions:\begin{enumerate}
    \item $\frac{\partial f}{\partial \lambda_i}>0$ for all $i=1,\ldots,n$, and $f>0$ on in $\Gamma$ and $f=0$ on $\partial \Gamma$;
    \item $f$ is concave homogeneous of degree 1 function; 
    \item $f(\lambda)\geq n\gamma_0^{\frac{1}{n}}(\Pi_{i=1}^{n}\lambda_i)^{\frac{1}{n}},$ for some positive constant $\gamma_0$ and for all $\lambda\in \Gamma_n$.
\end{enumerate}  
The basic example is  $f(\lambda)=\sigma_k^{\frac{1}{k}}(\lambda).$
 The determinant domination condition is motivated by the important work of Guo, Phong, and Tong \cite{MR4593734} and it's proved by Harvey and Lawson \cite{MR4589710} that there are many examples of such function $f$. 
Let $\chi$ be a real $(1,1)$-form on $M$, for any real function $\varphi$, we use $\chi_\varphi$ to denote $\chi+\ddc\varphi$. Set $\lambda[\chi_\varphi]\in \RR^n$ to be the eigenvalues of $\chi_\varphi$ with respect to $\omega_M$. In local coordiantes, if $\chi=\chi_{i\bar{j}}\sqrt{-1}dz_i\wedge d\overline{z_{j}}$ and $\omega_M=g_{Mi\bar{j}}\sqrt{-1}dz_i\wedge d\overline{z_{j}}$, then $\lambda[\chi]$ is the eigenvalus of $A^{i}_{j}:=g^{i\bar{k}}\chi_{j\bar{k}}$. We assume that $\lambda[\chi]\in \Gamma$. Let $F$ be a smooth function on $M$ with the normalization $\int_Me^{nF}\omega_M^n=1$. The equation we consider is the following \begin{equation}\label{equation}\begin{aligned}&f(\lambda[\chi_\varphi])=e^{c+F},\\
&\lambda[\chi_\varphi]\in \Gamma,\sup\varphi=0.
\end{aligned}
\end{equation}

The solvability of this equation for a smooth function $F$ can be found in \cite{CX24}  which is essentially based on the $C^{2,\alpha}$-estimate established by Sz\'ekelyhidi in \cite{MR3807322}. For given $F$, the uniqueness of solution $\varphi$ and constant $c$ is an easy consequence of the maximum principle.  Moreover, the constant $c$ is characterized by Guo and Song \cite{GS24}, where the notation of sup-slopes for a class of fully nonlinear equations was introduced. Specifically, they defined the sup-slope for a smooth function $F$ with respect to equation \eqref{equation} to be \[
    \sigma_F=\inf_{\varphi\in \cE}\max_{M}e^{-F}f(\lambda[\chi_\varphi]),
\]where $\cE:=\{u\in C^\infty(M,\RR)|\lambda[\chi_\varphi]\in \Gamma\}.$ Note that for any given smooth $F$, the sup-slope is always positive. This is because, for any function $\varphi\in \cE$, at the point where $\varphi$ attains its minimum, we have $\lambda[\chi_\varphi]-\lambda[\chi]\in \Gamma_n$. So $\sigma_F\geq e^{-\max F}f(\lambda[\chi])>0$. The theorem of Guo and Song in \cite{GS24} implied that the constant $e^c$ in equation \eqref{equation} equals $\sigma_F$. It provides a min-max characterization of the constant $e^c$ appearing in solving the equation Hermitian manifolds. The refined ABP estimate can be applied to get an estimate of the sup-slope as \begin{corollary}\label{supslopeestimate}
    Let $\varphi$ be a $C^2$-solution to equation \eqref{equation} and suppose $f$ satisfies the structural conditions stated above and $\chi$ is nondegenerate, i.e. there exists a constant $c_0>0$ such that $\lambda[\chi-c_0\omega_M]\in \Gamma$. Let $\Phi:\RR\rightarrow \RR^+$ be an increasing function with $\int_0^\infty\Phi^{-\frac{1}{n}}<\infty$. Assume that $\int_M\Phi(F)e^{nF}\omega_M^n\leq K$. Then there exists a constant $\epsilon_0>0$ depending on $M,\omega_M,n, \Phi, K,c_0,\gamma_0$ such that \[
    \sigma_F=e^c>\epsilon_0.
    \]
\end{corollary}

 As another application of the refined ABP estimate in the complex setting, we derive a gradient estimate for the complex Monge-Ampère equation. In Yau’s seminal work on solving Calabi’s conjecture \cite{https://doi.org/10.1002/cpa.3160310304}, the gradient estimate was bypassed using $C^0$ and $C^2$-estimates. A direct proof of the gradient bound was later provided by Blocki \cite{MR2495772}, assuming pointwise upper bounds on $F$ and $|\nabla F|$. In the work of Chen and He \cite{MR2993005}, they used the Moser iteration method to establish estimates up to $C^2$, assuming a $W^{1,p}$-bound on $F$ for some $p > 2n$. More recently, Chen and Cheng \cite{MR4028264} showed the gradient estimate by imposing conditions on the modulus of continuity of $F$.
In the recent work of Guo, Phong, and Tong \cite{MR4693954}, they derived the gradient estimate using the ABP estimate, assuming an upper bound on $\int_M |\nabla F|^{2n} e^{2F} \omega^n$ and a pointwise upper bound on $F$. Moreover, based on a deep theorem on uniform estimates of Green’s functions by Guo, Phong, and Sturm \cite{guo2022greens}, they reduced the pointwise bound on $F$ to an $L^p$-bound on $e^F$ for some $p > 1$ and an integral bound on $\int_M |\nabla F|^{q} e^F \omega^n$ for some $q > n$. In their paper, they also gave counterexamples when $q<n$. In this paper, we address the critical case where $q = n$ and prove the following:
\begin{theorem}Given $q>n$, $p>1$. Let $\varphi$ be a smooth solution to equation \eqref{CMAeq} with the conditions stated in section 3, we have the following estimate on the gradient of $\varphi$\[
	|\nabla\varphi|_{\omega}^2\leq C_1e^{C_2(\varphi-\inf_M\varphi)},
\]where $C_1, C_2$ are positive constants depending on a lower bound of the bisectional curvature of $\omega$ and $\omega,n,||e^F||_{L^q(\omega^n)},\int_M|\nabla F|_\omega^n\log^p(|\nabla F|_{\omega}+1)e^F\omega^n.$
\end{theorem}

The paper is organized as follows. In section 2, we prove the refined ABP estimates and give some direct corollaries. In section 3 we prove an application to gradient estimates of Monge-Ampere equations. In section 3 we prove a parabolic version of ABP estimates.

\textbf{Acknowledgements}. The author would like to thank his advisor, Professor Xiuxiong Chen, for his guidance and encouragement. He also thanks Jingrui Cheng for reading the draft and providing helpful comments and suggestions. The author was partially supported by Simons Foundation International, LTD.

\section{Complex version of ABP-estimate}
In this section, we prove an ABP-type estimate in the complex setting. Let $B_1\subset \mathbb{C}^n$ be the unit ball. We use $dV$ to denote the standard Euclidean volume form. 

\begin{proof}[Proof of theorem \ref{refinedABP}.]
Without loss of generality, we can assume that $\sup_{\partial B_1}u=0$.
    Let $e^{G_k}$ be a sequence of positive smooth functions approximating $\chi_{\Gamma^-}\det(-u_{i\bar{j}})$ from above uniformly, where $\chi_{\Gamma^-}$ is the characteristic function with respect to set $\Gamma^-$. Then it's easy to see that \begin{equation}\label{convergenceN}
        N_k:=\int_{B_1}e^{G_k}\Phi({G_k})dV\rightarrow \int_{\Gamma^-}\det(-u_{i\bar{j}})\Phi(\log(\det(-u_{i\bar{j}})))dV.
    \end{equation}
    Consider the following auxiliary function \begin{equation}\label{psi1construction}
(\ddc\psi_{1})^n=\frac{e^{G_k}\Phi(G_k)dV}{N_k}, \ddc\psi_{1}\geq 0, \psi_{1}|_{\partial B_1}=0.
    \end{equation}
    
For any given $q>1$,  which will be chosen later, define the function $h(s): \mathbb{R}\rightarrow \mathbb{R}$ as \[
    h(s)=-\int_s^\infty\frac{q}{\alpha}\frac{N_k^{\frac{1}{n}}}{\Phi^{\frac{1}{n}}(t)}dt,
\] where $\alpha>0$ is the constant in lemma \ref{alphainvariant}.
Since we assume that $\Phi$ is an increasing function, we know $h$ is a concave increasing function. 

Now define the function $\psi$ by $\psi=-h(-\frac{\alpha}{q}\psi_1)$. Then \[
	\ddc\psi=\frac{\alpha}{q}h'(-\frac{\alpha}{q}\psi_1)\ddc\psi_1-\frac{\alpha^2}{q^2}h''(-\frac{\alpha}{q})\sqrt{-1}\partial\psi_1\wedge\bar{\partial}\psi_1\geq \frac{\alpha}{q}h'(-\frac{\alpha}{q}\psi_1)\ddc\psi_1.
\]
Taking $n$-th power we have \[
	(\ddc\psi)^n\geq \frac{\alpha^n}{q^{n}}(h'(-\frac{\alpha}{q}\psi_1))^n(\ddc\psi_1)^n=\frac{\alpha^n}{q^n}h^{'n}\frac{e^{G_k}\Phi(G_k)dV}{N_k},
\] which is equivalent to \[
	e^{G_k}dV\leq \left(\frac{\alpha^n}{q^n}h^{'n}\frac{\Phi(G_k)}{N_k}\right)^{-1}(\ddc\psi)^n
\]
There are two possible cases, $\left(\frac{\alpha^n}{q^n}h^{'n}\frac{\Phi(G_k)}{N_k}\right)\geq 1$ or $\left(\frac{\alpha^n}{q^n}h^{'n}\frac{\Phi(G_k)}{N_k}\right)\leq 1$. 
In the first case, we have $e^{G_k}dV\leq (\ddc\psi)^n$, and in the second case, we have \[
	\Phi(G_k)\leq \frac{q^n}{\alpha^n}\frac{1}{h'^n(-\alpha q^{-1}\psi_1)}N_k. 
\]
According to our definition of $h$, we have $\Phi(s)=\frac{q^n}{\alpha^n}\frac{1}{h'^n(s)}N_k. $ Since $\Phi$ is an increasing function, in the second case, we have \[
G_k\leq -\frac{\alpha}{q}\psi_1.
\]
In summary, we have \begin{equation}\label{comparison}e^{G_k}dV\leq (\ddc\psi)^n+FdV.\end{equation} with $F=\min(\exp\{-\frac{\alpha}{q}\psi_1\},e^{G_k}).$

Next, we consider the auxiliary complex Monge-Ampere equation \[
	(\ddc\psi_{s,j})^n=\frac{\eta_j(u+\psi-s)FdV}{A_{s,j}}, \quad \psi_{s,j}|_{\partial B_1}=0, \text{ and } \ddc\psi_{s,j}>0,
\]
where $\eta_j$ is a sequence of smooth positive functions approximating $\max\{x,0\}$ from above uniformly and $A_{s,j}$ is the normalization constant defined by \[
	A_{s,j}:=\int_{B_1}\eta_{j}(u+\psi-s)FdV.
\] Then \[
\lim_{j\rightarrow \infty}A_{s,j}=A_s:=\int_{\Omega_s}(u+\psi-s)FdV.
\]
Consider the function \[
	\Psi:=u+\psi-s-\epsilon(-\psi_{s,j})^{-\frac{n}{n+1}}, 
\]where $\epsilon=\left(\frac{n+1}{n}\right)^{\frac{n}{n+1}}A_{s,j}^{\frac{1}{n+1}}.$

We claim that $\Psi\leq 0 $ for $s\geq s_0:=-h(0)<q\alpha^{-1}N_k^{\frac{1}{n}}\Lambda$. Firstly, on the boundary, we have $u\leq 0, \psi=-h(0)$ and $\psi_{s,j}=0$, so $\Psi\leq 0$ on the boundary when $s\geq -h(0)$. Without loss of generality, we can assume that the maximum of $\Psi$ is attained at $x_0$ for some $x_0$ in $\Omega_s:=\{u+\psi-s>0\}\subset \{u>0\}$. Then, at $x_0$, we have \[
	\begin{aligned}
	0\geq &\ddc\Psi=\ddc u+\ddc\psi+\epsilon\frac{n}{n+1}(-\psi_{s,j})^{-\frac{1}{n+1}}\ddc\psi_{s,j}\\
	&+\epsilon\frac{n}{(n+1)^n}(-\psi_{s,j})^{-\frac{n+2}{n+1}}\sqrt{-1}\partial\psi_{s,j}\wedge\bar{\partial}\psi_{s,j}\\
	\geq& \ddc u+\ddc\psi+\epsilon\frac{n}{n+1}(-\psi_{s,j})^{-\frac{1}{n+1}}\ddc\psi_{s,j}.
	\end{aligned}
\]
Therefore, we have $-\ddc u>0$ at this point and \[
\begin{aligned}
(-\ddc u)^n\geq& (\ddc\psi)^n +\left(\epsilon\frac{n}{n+1}\right)^n(-\psi_{s,j})^{-\frac{n}{n+1}}(\ddc\psi_2)^n
\end{aligned}
\]
Combine with equation \eqref{comparison} we have \[\begin{aligned}
   (\ddc\psi)^n+FdV\geq e^{G_k}\geq &(-\ddc u)^n\\
   \geq& (\ddc\psi)^n+\left(\epsilon\frac{n}{n+1}\right)^{n}(-\psi_{s,j})^{-\frac{n}{n+1}}(\ddc\psi_{s,j})^n. 
\end{aligned}
\]Plug in the definition of $\psi_{s,j}$ we get \[
	\eta_j(u+\psi-s)\leq A_{s,j}\left(\epsilon\frac{n}{n+1}\right)^{-n}(-\psi_{s,j})^{\frac{n}{n+1}}.
\]
From the definition of $\epsilon, $ and the fact $\eta_j\geq x_+$, the above inequality implies the claim. 

Now, we begin the iteration process. To begin with, we need a lemma from Kolodziej. 
\begin{lemma}[Kolodziej\cite{MR1618325} lemma 2.5.1]\label{alphainvariant}Let $\Omega\subset \mathbb{C}^n$ be a bounded pseudoconvex domain. There exists a constant $\alpha>0$ and $C>0$ such that for any $\psi$, a plurisubharmonic function on $\Omega$ with $\psi|_{\partial D}=0$ and unit Monge-Ampere measure $\int_{\Omega}(\ddc\psi)^n\leq1$, we have \[
	\int_{\Omega}e^{-\alpha\psi}\leq C.
\]
\end{lemma}
Let $\phi(s):=\int_{\Omega_s}FdV$. Then on $\Omega_{s+t}$, we have $(u+\psi-s)_+\geq t$. It follows that \begin{equation}\label{tphisplust}
	t\phi(s+t)\leq \int_{\Omega_{s}}(u+\psi-s)FdV=A_s.
\end{equation}
On the other hand, for $s\geq s_0$,\[
	\begin{aligned}
	A_s=&\int_{\Omega_s}(u+\psi-s)FdV\\
	\leq &c_nA_{s,j}^{\frac{1}{n+1}}\int_{\Omega_s}(-\psi_{s,j})^{\frac{n}{n+1}}FdV\\
	\leq &c_nA_{s,j}^{\frac{1}{n+1}}\left(\int_{\Omega_s}(-\psi_{s,j})^{\frac{np}{n+1}}FdV\right)^{\frac{1}{p}}\left(\int_{\Omega_s}FdV\right)^{1-\frac{1}{p}}.
	\end{aligned}
\]
Let $j\rightarrow \infty$, we have \[
	A_s\leq c(n,p)\left(\int_{\Omega_s}(-\psi_{s,j})^{\frac{np}{n+1}}FdV\right)^{\frac{n+1}{np}}\left(\int_{\Omega_s}FdV\right)^{(1-\frac{1}{p})\frac{n+1}{n}}
\]
We fix $p>n+1$, then $\delta:=(1-\frac{1}{p})\frac{n+1}{n}-1>0.$ Note that, with $\frac{1}{q^*}+\frac{1}{q}=1,$\[
	\int_{\Omega_s}(-\psi_{s,j})^{\frac{np}{n+1}}FdV\leq\left(\int_{B_1}(-\psi_{s,j})^{\frac{npq^*}{n+1}}\right)^{1/q^*}||F||_{L^q}\leq C(n).
\] In the above inequality we used lemma \ref{alphainvariant} and $F\leq e^{-\frac{\alpha}{q}\psi_1}$.

So we get 
\[
	A_s\leq c(n)\phi(s)^{1+\delta}.
\]
Combined with \eqref{tphisplust}, we get \[
	t\phi(s+t)\leq c(n)\phi(s)^{1+\delta}, \quad \forall s\geq s_0,\text{ and } t>0
\]
Then by the De Giorgi iteration lemma, we get \[
	\phi(s)=0, \quad \text{ for all } s\geq s_\infty:=\frac{2c(n)\phi(s_0)^\delta}{1-2^{-\delta}}+s_0.
\]
Choose $c_1=\frac{2c(n)\phi(s_0)^\delta}{1-2^{-\delta}}\leq c(n)(\int_{\Omega_0}FdV)^{\delta}\leq c(n)||e^{-\frac{\alpha}{q}\psi_1}||_{L^q}\leq c(n).$
Note also that $F\leq e^{G_k}$, we have 
$c_1\leq c(n)(\int_{\Omega_0} e^{G_k}dV)^{\delta}$, let $k\rightarrow \infty$,  we get \[c_1\leq c(n)\left(\int_{\Gamma^-}\det(-u_{i\bar{j}})dV\right)^{\delta}.\]

We remark the constant $\delta$ can be choose close but strictly less than $\frac{1}{n}.$

Let $k\rightarrow \infty$, and by  \eqref{convergenceN}, we see $s_0$ is bounded by $$c(n,\alpha,\Lambda)\left( \int_{\Gamma^-}\det(-u_{i\bar{j}})\Phi(\log(\det(-u_{i\bar{j}})))dV \right)^{\frac{1}{n}}.$$
Therefore \[
u+\psi\leq c_1+c_2\left( \int_{\Gamma^-}\det(-u_{i\bar{j}})\Phi(\log(\det(-u_{i\bar{j}})))dV \right)^{\frac{1}{n}},
\] with $c_1.c_2$ described in the theorem.
Note that from the construction of auxiliary function, $\psi$ is nonnegative, so the proof is complete. 

\begin{lemma}[De Giorgi]\label{DeGiorgiiterationlemma}Let $\phi:\RR\rightarrow \RR$ be a monotone decreasing function such for some $\delta>0$ and any $s\geq s_0,t>0,$\[
   t\phi(t+s)\leq C_0\phi(s)^{1+\delta}. 
\]
Then $\phi(s)=0$ for any $s\geq \frac{2C_0\phi^{\delta}(s_0)}{1-2^{-\delta}}+s_0$.
\end{lemma}
\end{proof}
An immediate corollary is the following uniform estimate for the Dirichlet problem of complex Monge-Ampere equations, established by Kolodziej, \begin{corollary}Let $u$ be a smooth solution to the Dirichlet problem \[
(\ddc u)^n=e^fdV, \quad \ddc u>0 \text{ in } B_1, \quad u|_{\partial B_1}=0.
\] Let $\Phi:\RR\rightarrow \RR^+$ be an increasing function such that $\int_0^\infty\Phi^{-\frac{1}{n}}\leq \Lambda<\infty$. Then  there exists a constant $C$ depending on $ n, \Lambda, \int_{B_1}e^f\Phi(f)dV$ such that \[
-C\leq u\leq 0.
\]
    
\end{corollary}

\begin{corollary}
    Let $a^{i\bar{j}}(x)$ be a smooth Hermitian matrix-valued function on $B_1\subset\CC^n$ which is semipositive definite everywhere. Assume $\Phi:\RR\rightarrow \RR^+$ is an positive increasing function with $\int_0^\infty\Phi^{-\frac{1}{n}}(t)dt\leq \Lambda<\infty$. Let $u\in C^{2}(B_1)\cap\ C^{0}(\overline{B_1})$ be a real function satisfying $a^{i\bar{j}}u_{i\bar{j}}\geq f$ for a smooth function $f$, then there exists a constant $c_1,c_2,\delta>0$ depending on $n,\Lambda$ such that \[
    \sup_{B_1}u\leq \sup_{\partial B_{1}}u+c_1+c_2\left(\int_{\Gamma_-}\frac{(f^-)^n}{\det(a^{i\bar{j}})}\Phi(\log(\frac{(f^-)^n}{n^n\det(a^{i\bar{j}})}))dV\right)^{\frac{1}{n}},
    \]
    where $\Gamma^-:=\{x\in B_1|u(x)>\sup_{\partial B_1}u, \ddc u(x)< 0\}.$
    Moreover the constant $c_1$ can be chosen to be \[
    c_1=\min\{c(n),\left(\int_{\Gamma^-}\det(-u_{i\bar{j}})\right)^{\delta}\}, \text{ for a dimensional constant } c(n).
    \]
\end{corollary}

\begin{proof}
    Apply the above ABP estimate with the inequality, at point $\ddc u\leq 0$, \[f^-\geq a^{i\bar{j}}(-u)_{i\bar{j}}\geq n\det(a^{i\bar{j}})^\frac{1}{n}\det(-u_{i\bar{j}})^{\frac{1}{n}}\].
\end{proof}

Another corollary (of the proof) is that when $\Phi$ grows slowly, which may not satisfy $\int^\infty\Phi^{-\frac{1}{n}}<\infty$, we still have integral estimates. We state a simple case as follows: 
\begin{corollary}
    Fix $0<p<n$, let $u\in C^2(B_1)\cap C^0(\overline{B_1})$. Assume that $u|_{\partial B_1}\leq 0$. Then there exists a uniform constant $c_1,c_2,c_3$ depending on $n,p$ such that \[
    \int_{B_1}\exp\left\{c_1\left(\frac{(u-c_2)_+}{\cN_p^\frac{1}{n}}\right)^{\frac{n}{n-p}}\right\}dV\leq c_3,
    \]where $\cN_p:=\int_{\Gamma^-}\det(-u_{i\bar{j}})(\log^2(\det(-u_{i\bar{j}}))+1)^\frac{p}{2}dV.$
\end{corollary}
\begin{proof}
    Choose $\Phi(x)=x^p$ when $x>2$, and define $h(s)=\int_0^s\frac{q}{\alpha}\frac{\cN_p^\frac{1}{n}}{\Phi(x)^{\frac{1}{n}}}dx$. Then a similar argument shows that $u+h(-\frac{\alpha}{q}\psi)\leq C$ for some uniform constant $C$ and $\psi_1$ as \eqref{psi1construction} . Note that $h(s)\leq C_1s^{1-\frac{p}{n}}+C_2$. It follows that $(u-C_3)_+\leq C_4\cN_p^{\frac{1}{n}}(-\psi_1)^{\frac{n-p}{n}}.$ Apply lemma \ref{alphainvariant}, we get the desired estimate. 
\end{proof}

Now we give the proof of corollary \ref{supslopeestimate}.
\begin{proof}[Proof of Corollary \ref{supslopeestimate}]
    Let $x_0$ be the point where $\varphi$ attains its minimum and choose a coordinates chart $B(x_0,r_0)$ such that on this coordinates chart we have $c_0\omega_M\geq b_0\ddc|z|^2$. Then one can apply the refined ABP estimate to function \[
    u=-(\varphi-\min\varphi)-b_0(|z|^2-r_0^2).
    \]
    Note that $u\leq 0$ on $\partial B(x_0,r_0)$ and $u(x_0)=b_0r_0^2$. Therefore\[
\begin{aligned}
     b_0r_0^2\leq& c_1\left(\int_{\Gamma^-}\det(\varphi_{i\bar{j}}+b_0\delta_{ij})\right)^{\delta_n}+c_2\left(\int_{\Gamma^-}\det(\varphi_{i\bar{j}}+\delta_{ij})\Phi(\log(\det(\varphi_{i\bar{j}}+\delta_{ij})))\right)^{\frac{1}{n}}\\
     \leq &c_1\left(\int_{\Gamma^-}\det(c_0\omega_{Mi\bar{j}}+\varphi_{i\bar{j}})\right)^{\delta_n}+c_2\left(\int_{\Gamma^-}\det(\varphi_{i\bar{j}}+c_0\omega_{Mi\bar{j}})\Phi(\log(\det(\varphi_{i\bar{j}}+c_0\omega_{Mi\bar{j}})))\right)^{\frac{1}{n}}\\
     \leq &C\left(\int_{\Gamma^-}f(\lambda[\chi_\varphi])^n\right)^{\delta_n}+C\left(\int_{\Gamma^-}f(\lambda[\chi_\varphi])^n\Phi(n\log(f(\lambda[\chi_\varphi])))\right)^{\frac{1}{n}}\\
     =&C\left(\int_{\Gamma^-}e^{nF+nc}\right)^{\delta_n}+C\left(\int_{\Gamma^-}e^{nF+nc}\Phi(nF+nc)\right)^{\frac{1}{n}}
\end{aligned}
\]In the third line we used the assumptions on $f$ as \[
f(\lambda[\chi+\ddc\varphi])\geq f(\lambda[\chi-c_0\omega_M])+f(\lambda[c_0\omega_M+\ddc\varphi])\geq c(n,\gamma_0)\det(c_0\omega_{Mi\bar{j}}+\varphi_{i\bar{j}})^{\frac{1}{n}}
\]
The above inequality implies a lower bound of the constant $c$ so the proof is completed. 
\end{proof}

\begin{remark}\begin{enumerate}
\item The proof works for bounded pseudoconvex domain $\Omega\subset \CC^n$.
\item If we assume that $\sup_{\partial B_1}u\leq 0$, and with the same assumption as above, then the inequality holds \[
	\sup_{B_1} u\leq C_1\left(\int_{\{u\geq 0\}}\frac{(f^-)^n}{n^n\mathscr{D}}dV\right)^{\delta}+C_2\left(\int_{\{u\geq 0\}}\frac{(f^-)^n}{n^n\mathscr{D}}\Phi(\log(\frac{(f^-)^n}{n^n\mathscr{D}}))dV\right)^{\frac{1}{n}}
\]

\item One can apply the above refined ABP-estimates to get sharp $L^\infty$-estimate of complex Monge-Ampere type equations on compact Hermitian manifold. We sketch a proof here for complex Monge-Ampere equations and $\Phi(x)=|x|^n\log^p(|x|+1),p>n$ for $x>0$, for simplicity. Let $(M,\omega)$ be a compact Hermitian manifold with $\omega$ a fixed Hermitian metric. Consider the equation $(\omega+\ddc\varphi)^n=e^F\omega^n,\omega+\ddc\varphi>0,\sup\varphi=-1.$  Assume that $\varphi$ attains its minimum at $x_0$. We choose a coordinate chats $z$ around $x_0$, say $B(x_0,r_0)$ such that $\frac{1}{2}\ddc|z|^2\leq \omega\leq 2\ddc|z|^2$. Define a function $u=-\varphi+\inf\varphi-\frac{1}{2}(|z|^2-r^2_0)$. Then $u\leq 0$ on $\partial{B_{r_0}}$ and $u(x_0)=\frac{1}{2}r_0^2$. Moreover, $u$ satisfies the differential inequality \[
	\Delta_\varphi u=-n+\tr_\varphi \omega-\frac{1}{2}\tr_\varphi\ddc|z|^2\geq -n.
\]Therefore by our refined ABP-estimate, for $p-2\epsilon>n$,  \begin{equation}\label{uniformlowerbound}
	\frac{1}{2}r_0^2=u(x_0)\leq \int_{\{u>0\}}e^FdV+c_2\left(\int_{\{u>0\}}e^F|F|^n(\log(|F|+1))^{p-2\epsilon}dV\right)^{\frac{1}{n}}.
\end{equation}
Note that on $\{u>0\}$, we have $-\inf\varphi-\frac{1}{2}r_0^2\leq -\varphi-\frac{1}{2}|z|^2$. It follows that \[\begin{aligned}
\int_{\{u>0\}}e^FdV\leq &\int_{B(x_0,r_0)}\frac{\frac{1}{2^n}\log^n(-\varphi-\frac{1}{2}|z|^2)}{\frac{1}{2^n}\log^n(-\inf\varphi-\frac{1}{2}r_0^2)}e^FdV\\
\leq &\frac{1}{\frac{1}{2^n}\log^n(-\inf\varphi-\frac{1}{2}r_0^2)}\left(\int_{B(x_0,r_0)}(-\varphi-\frac{1}{2}|z|^2)dV+\int_{B(x_0,r_0)}e^F(|F|+1)^ndV\right)\\
\leq &C\frac{1}{\frac{1}{2^n}\log^n(-\inf\varphi-\frac{1}{2}r_0^2)}.
\end{aligned}
\]
We have used the elementary inequality $xe^y\leq e^y\log^n(|y|+1)+C_ne^{2x^{\frac{1}{n}}}$ in the second line for $x=\frac{1}{2^n}\log^n(-\varphi-\frac{1}{2}|z|^2),y=F$. Similarly, using $xe^y\leq e^y\log^{\epsilon}(\log(|y|+1)+1)+ x\exp\exp\exp x^{\frac{1}{\epsilon}}\leq e^y\log^{\epsilon}(\log(|y|+1)+1)+ C_\epsilon\exp\exp\exp 2x^{\frac{1}{\epsilon}}$, for $x=\frac{1}{2^\epsilon}\log(\log(\log(-\varphi-\frac{1}{2}|z|^2)))$(one can renormalize $\varphi$ such that $x>0$) and $y=\log(e^F|F|^n\log^{p-2\epsilon}(|F|+1)),$ we can estimate the second term as \[\begin{aligned}
	&\int_{\{u>0\}}e^F|F|^n(\log(|F|+1))^{p-2\epsilon}dV \\
	\leq &\frac{\int_{B(x_0,r_0)}(-\varphi-\frac{1}{2}|z|^2)dV+\int_{B(x_0,r_0)}e^F|F|^n\log^{p}(|F|+1))dV}{\frac{1}{2^\epsilon}\log(\log(\log(-\inf\varphi-\frac{1}{2}r_0^2)))}
\end{aligned}
\]
Together with (\ref{uniformlowerbound}), we get an uniform upper bound for $-\inf\varphi$. 
\end{enumerate}
\end{remark}
\section{Application to gradient estimate}
In this section, Let $(M,\omega)$ be a compact K\"ahler manifold with a given K\"ahler metric $\omega$. Let $F$ be a smooth function on $M$ with normalization $\int_Me^F\omega^n=\int_M\omega^n$. We consider the following complex Monge-Ampere equation on $M$:\begin{equation}\label{CMAeq}
	(\omega+\ddc\varphi)^n=e^F\omega^n.
\end{equation}
\begin{theorem}Given $p>n$, $q>1$. Let $\varphi$ be a smooth solution to equation \eqref{CMAeq} with the above conditions, we have the following estimate on the gradient of $\varphi$\[
	|\nabla\varphi|_{\omega}^2\leq C_1e^{C_2(\varphi-\inf_M\varphi)},
\]where $C_1,C_2$ are positive constants depending on a lower bound of the bisectional curvature of $\omega$ and $\omega,n,||e^F||_{L^q(\omega^n)},\int_M|\nabla F|_\omega^n\log^p(|\nabla F|_{\omega}+1)e^F.$
\end{theorem}

\begin{remark}Compared to Guo-Phong-Tong's results, where they use the classical ABP-estimate from real setting directly, we lower the exponent to be almost sharp. And compared to the results of Guo-Phong-Sturm, where they use the estimate of Green's function, we don't require a uniform lower bound of the volume form (very recently, they extended their results without such a lower bound in \cite{Guo24ii}).
\end{remark}
\begin{proof}
Without loss of generality, we can assume $\inf_M\varphi=0$. 
Let $H:=e^{-\lambda\varphi}|\nabla\varphi|_\omega^2$ for $\lambda=2B+100A$ where $-B$ is the lower bound of the bisectional curvature of $(M,\omega)$ and $A$ to be determined later. The following lemma is standard, c.f.\cite{MR4693954}.
\begin{lemma}We have the following differential inequality\[
	\Delta_\varphi H\geq 2e^{-\lambda\varphi}\nabla F\cdot_\omega\nabla\varphi+(\lambda-2B)H\tr_{\omega_\varphi}\omega-\lambda(n+2)H.
\]
\end{lemma}
With the above lemma, we have \[\begin{aligned}
\Delta_\varphi H^2=&2H\Delta_\varphi H+2|\nabla H|_{\omega_\varphi}^2\\
\geq& -4e^{-\frac{\lambda}{2}\varphi}H^{\frac{3}{2}}|\nabla F|_\omega+2(\lambda-2B)H^2\tr_{\omega_\varphi}\omega-2\lambda(n+2)H^2+2|\nabla H|^2_{\omega_\varphi}.
\end{aligned}
\] Assume $H$ attains its maximum at $x_0$.  
Let $\eta$ be a cut-off function on $B(x_0,r)$ with \[\eta=1 \text{ on } B(x_0,\frac{r}{2}), \text{ and }\eta=1-\theta \text{ on } M\setminus B(x_0,\frac{3r}{4}),\] and \[
	|\nabla\eta|_\omega\leq A\frac{\theta}{r}, \quad |\nabla^2 \eta|_{\omega}\leq \frac{A\theta}{r^2}.
\]The constant $A$ depends only on $M,\omega$.

Then, \[
	\begin{aligned}
	\Delta_\varphi(\eta H^2)=&\eta\Delta_\varphi H^2+4H\nabla H\cdot_{\omega_\varphi}\nabla\eta+H^2\Delta_\varphi\eta\\
	\geq &-4\eta e^{-\frac{\lambda}{2}\varphi}H^{\frac{3}{2}}|\nabla F|_\omega+2(\lambda-2B)\eta H^2\tr_{\omega_\varphi}\omega-2\lambda(n+2)\eta H^2+2\eta|\nabla H|^2_{\omega_\varphi}\\
	& -4H|\nabla H|_{\omega_\varphi}|\nabla \eta|_{\omega_\varphi}-H^2|\nabla^2\eta|_{\omega}\tr_{\omega_\varphi}\omega\\
	\geq &-4\eta e^{-\frac{\lambda}{2}\varphi}H^{\frac{3}{2}}|\nabla F|_\omega+2(\lambda-2B)\eta H^2\tr_{\omega_\varphi}\omega-2\lambda(n+2)\eta H^2+2\eta|\nabla H|_{\omega_\varphi}^2\\
	&-2\eta|\nabla H|_{\omega_\varphi}^2-8H^2\eta^{-1}|\nabla\eta|_{\omega_\varphi}^2-H^2|\nabla^2\eta|_{\omega}\tr_{\omega_\varphi}\omega\\
	\geq &-4\eta e^{-\frac{\lambda}{2}\varphi}H^{\frac{3}{2}}|\nabla F|_\omega+(2(\lambda-2B)\eta-8\eta^{-1}|\nabla\eta|^2_{\omega}-|\nabla^2\eta|_{\omega})H^2\tr_{\omega_\varphi}\omega-2\lambda(n+2)\eta H^2\\
	\geq &-4\eta e^{-\frac{\lambda}{2}\varphi}H^{\frac{3}{2}}|\nabla F|_{\omega}-2\lambda(n+2)\eta H^2.
	\end{aligned}
\]
Let $M$ be the maximum value of $\eta H$, then by the complex version of ABP-estimate, we have \[
	M^2\leq (1-\theta)M^2+c_1+c_2\cN_\Phi^{\frac{1}{n}}((4\eta e^{-\lambda\varphi}H^{\frac{3}{2}}|\nabla F|_{\omega}+2\lambda(n+2)\eta H^2)^ne^F)
\]
Here we use the notation of entropy $\cN_\Phi(e^G):=\int e^G\Phi(G)$. 
Let $\Phi(x)=|x|^{p_1}$ for $p_1>n$. We estimate the entropy term carefully in the following,\[
	\begin{aligned}
	&\cN_\Phi((4\eta e^{-\frac{\lambda}{2}\varphi}H^{\frac{3}{2}}|\nabla F|_{\omega}+2\lambda(n+2)\eta H^2)^ne^F)dV\\
	=&\int_{B_r}((4\eta e^{-\frac{\lambda}{2}\varphi}H^{\frac{3}{2}}|\nabla F|_{\omega}+2\lambda(n+2)\eta H^2)^ne^F\left(|n\log((4\eta e^{-\frac{\lambda}{2}\varphi}H^{\frac{3}{2}}|\nabla F|_{\omega}+2\lambda(n+2)\eta H^2))+F|\right)^{p_1}dV\\
	\leq &C(n,p_1)\int_{B_r}(e^{-\frac{\lambda n}{2}\varphi}H^{\frac{3n}{2}}|\nabla F|_{\omega}^n+H^{2n})e^F(\log^
	{p_1}(e^{-\lambda\varphi}H^{\frac{3}{2}}|\nabla F|_\omega+H^2)+|F|^{p_1}+1)dV\\
	\leq& C(n,p_1)\int_{B_r}e^{-\frac{\lambda n}{2}\varphi}H^{\frac{3}{2}}|\nabla F|_{\omega}^ne^F(\log^{p_1}(1+e^{-\frac{\lambda}{2}\varphi}H^{\frac{3}{2}})+\log^{p_1}(1+|\nabla F|_\omega)+\log^{p_1}(H^2+1)+|F|^{p_1}+1)dV\\
	&+\int_{B_r}H^{2n}e^F(\log^{p_1}(1+e^{-\frac{\lambda}{2}\varphi}H^{\frac{3}{2}})+\log^{p_1}(1+|\nabla F|_\omega)+\log^{p_1}(H^2+1)+|F|^{p_1}+1)dV
	\end{aligned}
\]
We use the elementary inequality $(x+y)^p\leq 2^px^p+2^py^p$ ,$\log(x+y)\leq \log(1+x)+\log(1+y)$ and $\log(1+xy)\leq \log(1+x)+\log(1+y)$ for any $x,y>0$ in the third and fourth lines.
Now estimate term by term after expanding the bracket. 
The first term can be estimated as \[
	\int_{B_r}e^{-\frac{\lambda}{2}\varphi}H^{\frac{3n}{2}}|\nabla F|_\omega^n e^F\log^{p_1}(1+e^{-\frac{\lambda}{2}\varphi}H^{\frac{3}{2}})dV\leq C(p_1)(M^{\frac{3n+\epsilon}{2}}+C_\epsilon)\int_{B_r}|\nabla F|^n_\omega e^FdV,
\]since for any $\epsilon>0$ there exists constant $C_\epsilon$ such that $\log(1+x)\leq x^\epsilon+C_\epsilon$ for any $x>0$.
For the second term, \[
	\int_{B_r}e^{-\frac{\lambda}{2}\varphi}H^{\frac{3n}{2}}|\nabla F|_\omega^ne^F\log^{p_1}(1+|\nabla F|_\omega)\leq CM^{\frac{3n}{2}}\int_{B_r}|\nabla F|^{n}\log^{p_1}(1+|\nabla F|_\omega^n)e^FdV.
\]
For the third term\[
		\int_{B_r}e^{-\frac{\lambda}{2}\varphi}H^{\frac{3n}{2}}|\nabla F|_\omega^n e^F\log^{p_1}(1+H^2)dV\leq C(p_1)(M^{\frac{3n+\epsilon}{2}}+C_\epsilon)\int_{B_r}|\nabla F|^n_\omega e^FdV.
\]
For the fourth term,\[
		\int_{B_r}e^{-\frac{\lambda}{2}\varphi}H^{\frac{3n}{2}}|\nabla F|^n_\omega e^F(|F|^{p_1}+1)dV\leq CM^{\frac{3n}{2}}\int_{B_r}(|\nabla F|_\omega^n\log^{q}(|\nabla F|^n+1)+|F|^{p_1}e^{|F|^{\frac{p_1}{q}}}+1)e^FdV,
		\]
		where we used the following inequality with $\epsilon=1, t=q.$ Therefore, if we choose $n<p_1<q$, we have $|F|^{p_1}e^{|F|^{\frac{p_1}{q}}}e^F\leq C_{\epsilon}e^{(1+\epsilon)F}$ for any $\epsilon>0$. Hence the fourth term is also uniformly bounded. 

\begin{lemma}\label{younginequality}
For any $x>0,y>1$, and any $t>0,\epsilon>0$, one has \[
    xy\leq \epsilon y(\log y)^t+xe^{(\epsilon^{-1} x)^{\frac{1}{t}}}
\]
\end{lemma}
\begin{proof}
It's equivalent to show \[x-\epsilon(\log)^{t}\leq x\frac{e^{(\epsilon^{-1}x)^{\frac{1}{t}}}}{y}.\]
If $x\leq \epsilon(\log y)^{t}$, i.e.$y\geq \exp{(\epsilon^{-1}x)^{\frac{1}{t}}}$, the LHS is less than 0, and the inequality is trivial. Otherwise, if $y<\exp{(\epsilon^{-1}x)^{\frac{1}{t}}}$, we have RHS is greater than $x$, and the inequality is also valid. 
\end{proof}
Go back to the proof, for the fifth term\[
\begin{aligned}
	&\int_{B_r}H^{2n}e^F(\log^{p_1}(1+e^{-\lambda\varphi}H^{\frac{3}{2}}))dV\\
	\leq &C(M^{2n+\epsilon-\delta}+C_\epsilon)\int_{B_r}H^{\delta}e^FdV\\
	\leq &C(M^{2n+\epsilon-\delta}+C_\epsilon)\left(\int_{B_r}HdV\right)^{\delta}\left(\int_{B_r}e^{\frac{F}{1-\delta}}dV\right)^{1-\delta}.
	\end{aligned}
\]Choose $\delta$ close to 1, and let $\epsilon<\delta$,  by assumption $||e^F||_{L^p}$ is uniformly bounded for some $p>1$, so the fifth term is bounded uniformly by $CM^{2n+\epsilon-\delta}$.

For the sixth term, \[
	\begin{aligned}
	&\int_{B_r}H^{2n}e^F\log^{p_1}(1+|\nabla F|_\omega)dV\\
	\leq &C\int_{B_r}\left(H^{2n}\log^{p_1}(1+H^{2n})e^F+(1+|\nabla F|_\omega)\log^{p_1}(1+|\nabla F|_\omega)e^F\right)dV\\
	\leq &C(M^{2n+\epsilon-\delta}+C_\epsilon)\left(\int_{B_r}HdV\right)^{\delta}\left(\int_{B_r}e^{\frac{F}{1-\delta}}dV\right)^{1-\delta}\\
	&+\int_{B_r}(1+|\nabla F|_\omega)\log^{p_1}(1+|\nabla F|_\omega)e^FdV.
	\end{aligned}
\]The term $\int_{B_r}(1+|\nabla F|_\omega)\log^{p_1}(1+|\nabla F|_\omega)e^FdV$ above  is uniformly bounded by assumptions.

For the seventh term, \[
	\int_MH^{2n}e^F\log^{p_1}(1+H^2)\leq C(M^{2n+\epsilon-\delta}+C_\epsilon)\left(\int_{B_r}HdV\right)^{\delta}\left(\int_{B_r}e^{\frac{F}{1-\delta}}dV\right)^{1-\delta}.
\]

For the last term \[
\begin{aligned}
	&\int_{B_r}H^{2n}e^F(|F|^p+1)dV\\
	\leq &\int_{B_r}H^{2n}\log^{\frac{1}{2p}}(1+H^{2n})e^FdV+\int_{B_r}(|F|^p+1)e^{(|F|^p+1)^{1/2p}}e^FdV\\
	\leq &C(M^{2n+\epsilon-\delta}+C_\epsilon)\left(\int_{B_r}HdV\right)^{\delta}\left(\int_{B_r}e^{\frac{F}{1-\delta}}dV\right)^{1-\delta}+C(p)\int_{B_r}e^{\frac{F}{1-\delta}}dV
\end{aligned}
\]
\begin{lemma}[Guo-Phong-Tong \cite{MR4693954}]There exists a constant $C$ depends on the parameters as in the theorem, such that $\int_MH\omega^n\leq C$.
\end{lemma}
In summary, we have \[
	M^2\leq (1-\theta)M^{2}+C(n,p,q,\omega,||e^F||_{L^{q}(\omega^n)},||\nabla F||_{L^n\log^p(e^F\omega^n)})(1+M^\frac{3+\epsilon}{2}+M^{\frac{3}{2}}+M^{\frac{2n+\epsilon-\delta}{n}}),
\] which implies an upper bound on $M$. 
\end{proof}

\section{Parabolic ABP-estimate}
In this section we prove a parabolic version of theorem \ref{refinedABP}. We define  $Q_T:=B_1(0)\times (0,T)$, $\partial_PQ_T:=B_1\times\{0\}\cup \partial B_1\times[0,T)$, the parabolic boundary of $Q_T$. Then we have 
\begin{theorem}Let $u\in C^{2,1}(Q_T)\cap C^0(\overline{Q_T})$ satisfy the differential inequality $(-\partial_t+a^{i\bar{j}}\partial_i\partial_{\bar{j}})u\geq f$ for a smooth function $f$. Let $\Phi(x)$ be an increasing positive function and $\int_0^\infty\Phi^{-\frac{1}{n+1}}(t)dt<\infty$. Then there exists a constant $C_1(n,T), C_2(n,\Phi)$ such that \[
	\sup_{Q_T}u\leq \sup_{\partial_PQ_T}u+C_1+C_2\cN_{\Phi}^{\frac{1}{n+1}}(\frac{(f^-)^{n+1}}{(n+1)^{n+1}\mathscr{D}}),
\]where $\mathscr{D}=\det(a_{i,\bar{j}})$, and $\cN_\Phi(\frac{(f^-)^{n+1}}{(n+1)^{n+1}\mathscr{D}})=\int_{Q_T}\Phi(\log(\frac{(f^-)^{n+1}}{(n+1)^{n+1}\mathscr{D}}))\frac{(f^-)^{n+1}}{(n+1)^{n+1}\mathscr{D}}dVdt$.

\end{theorem}

We need a lemma to construct the auxiliary comparison function. 
\begin{lemma}\label{existenceofinvermongeampereflow}Given any $T>0$, let $\varphi_0=|z|^2-1$, and $f\in C^\infty(\overline{Q_T})$ be a smooth function and $f>0$ in $\overline{Q_T}$. $b(t)\geq 0$ is a nonnegative strict increasing function, i.e. $b'(t)>0$ for $t\in [0,T]$ and $b(0)=0$. Then the equation \begin{equation}\label{auxiliarylocalflow}
	(-\partial_t u)\det(u_{i\bar{j}})=f, \quad u(0)=\varphi_0, \quad u|_{\partial B_1}=-b(t), \quad \ddc u>0.
\end{equation} has an  unique solution $u\in C^\infty(Q_T)\cap C^{2,1}(\overline{Q_T})$ if the following compatibility condition is satisfied \begin{equation}\label{compatibilitycondition}
	b'(0)\det((\varphi_0)_{i\bar{j}})=f(x,0), \quad x\in \partial{B_1}.
\end{equation} 
\end{lemma}
We leave the proof of this lemma in the appendix.

To get a exponential integrability proposition like lemma \ref{alphainvariant}, we consider the functional \[
	\cE(\varphi):=\int_{B_1}(-\varphi)(\ddc\varphi)^n, \quad \varphi\in \PSH_0(B_1).
\]
Here, we use the notation $\PSH_0$ to denote the set of $\PSH$-functions which vanishes on the boundary $\partial B_1$.
Then direct computation shows that \[
\begin{aligned}
\delta \cE(\varphi)(-\psi)=&\int_{B_1}(-\psi)(\ddc\varphi)^n+(-\varphi)n\ddc\psi\wedge(\ddc\varphi)^{n-1}\\
=&\int_{B_1}(-\psi)(\ddc\varphi)^n+\int_{B_1}(-\psi)n\ddc\varphi\wedge(\ddc\varphi)^{n-1}\\
=&(n+1)\int_{B_1}(-\psi)(\ddc\varphi)^n.
\end{aligned}
\]
It follows that along the flow \eqref{auxiliarylocalflow}, we have \[
	\frac{\partial}{\partial t}\cE(u(t)+b(t))=(n+1)\int_{B_1}(-u_t-b'(t))(\ddc u)^n\leq (n+1)\int_{B_1}fdVdt,
\]since we assume that $b'(t)\geq 0.$
We get the following  \begin{lemma}
Let $u$ be a smooth solution to equation \eqref{auxiliarylocalflow}, then for any $t\in[0,T]$, we have \[
	\cE(u(t)+b(t))\leq \cE(\varphi_0)+(n+1)\int_0^t\int_{B_1}fdVdt.
\]
\end{lemma}
\begin{corollary}\label{parabolicalphainvariant}There exist uniform constants $\alpha, C$ depending on $n,\int_0^T\int_{B_1}fdVdt,\cE(\varphi_0)$ such that \[
	\int_{B_1}e^{-\alpha u(t)}dV\leq Ce^{\alpha b(t)}.
\]
\end{corollary}
\begin{proof}
Let $\psi$ be the solution to the complex Monge-Ampere equation\[
	(\ddc\psi)^n=\frac{(-u-b(t))(\ddc u)^n}{\cE(u+b(t))},\quad \varphi|_{\partial B_1}=0, \text{ and } \ddc\psi>0.
\]
Consider the function  \[
	\Psi:=-(u+b(t))-A(-\psi)^{\frac{n}{n+1}}.
\]We claim that $\Psi\leq 0$. 
Since $u|_{\partial B_1}=0$, we can assume that maximum is attained at interior of $\partial B_1$, say $x_0$. Then at $x_0$, we have \[
	0\geq -\ddc u+A\frac{n}{n+1}(-\psi)^{-\frac{1}{n+1}}\ddc\psi
\]Taking $n$-th power we get \[
	(\ddc u)^n\geq\left(\frac{A n}{n+1}\right)^{n}(-\psi)^{-\frac{n}{n+1}}\frac{(-u-b(t))(\ddc u)^n}{\cE(u+b(t))}.
\]By choice of $A=\left(\frac{n+1}{n}\right)^{\frac{n}{n+1}}\cE(u+b(t))^{\frac{1}{n+1}},$ the above inequality is equivalent to $\Psi\leq 0.$ Then by lemma \ref{alphainvariant}, we get \[
	\int_{B_1}\exp\{\alpha\frac{n}{n+1}\frac{(-u-b(t))^{\frac{n+1}{n}}}{\cE^{\frac{1}{n}}(u+b(t))}\}\leq C.
\]
\end{proof}
\begin{remark}\begin{enumerate}
\item The above Moser-Trudinger type estimate was also proved by Wang-Wang-Zhou in \cite{MR4155289} by the local Monge-Ampere flow. We give a short proof combining the idea of Chen-Cheng that compare the function with the solution to an auxiliary complex Monge-Ampere equation. Here we admit the  Kolodziej's lemma \ref{alphainvariant} while  it's also proved in Wang-Wang-zhou's paper by a pure PDE method.
\item For any given $T$, and $f$ which equals to a positive constant $c$ on the vertical boundary $\partial B_1\times[0,T]$, we can find a strictly increasing function $b$ such that $b$ satisfies the compatibility condition \eqref{compatibilitycondition} and $b(t)\leq 1$ for all $t\in [0,T].$ It means that for any $f$ which equals to a positive constant, one can adjust the boundary condition $u(x,t)=-b(t)$ to ensure that along the flow we have \[
	\int_{B_1}e^{-\alpha u(t)}dV\leq C, \text{ uniformly. }
\]
\end{enumerate}
\end{remark}

Firstly, we let $\psi$ be the solution to \[
	(-\partial_t\psi_1)(\ddc \psi_1)^n=\frac{e^G\Phi(G)}{\cN_\Phi(G)}, \quad \psi_1=-b_G(t) \text{ on } \partial B_1, \quad \psi_1|_{t=0}=\varphi_0,
\]
where $\Phi(s)$ is a given increasing function mapping $\RR$ to $\RR$, $G$ is a smooth function being constant on $\partial B_1\times [0,T]$, $b_G(t)$ a strict increasing function satisfying the compatibility condition and $b(t)\leq 1$ and the Nash entropy is defined by \[
	\cN_\Phi(G):=\int_0^T\int_{B_1}G\Phi(\log G)dVdt.
\] 
The existence of solution to such equation guaranteed by lemma \ref{existenceofinvermongeampereflow}. 
Given any $q>1$, we define the function $h$ as \[
	h(s)=-\int_0^s\frac{q}{\alpha}\frac{\cN^{\frac{1}{n+1}}_\Phi(G)}{\Phi^{\frac{1}{n+1}}(t)}dt.
\]
Then a similar computation as in the elliptic case shows that \begin{equation}\label{paraboliccomparison}
e^G\leq (-\partial_t\psi)(\ddc\psi)^n/dV+e^{-\frac{\alpha}{q}\psi_1}.
\end{equation}
Moreover, if we let $F=\min\{e^{-\frac{\alpha}{q}\psi_1},e^G\}$, we have \[
	e^G\leq  (-\partial_t\psi)(\ddc\psi)^n/dV+F
\]

From the corollary \ref{parabolicalphainvariant}, we know that $e^{-\frac{\alpha}{q}\psi_1}$ is uniformly bounded in $L^q$.

Now we consider another auxiliary equation \[
	(-\partial_t\psi_{s,j})(\ddc\psi_{s,j})=\frac{\eta_j(u+\psi-s)FdV}{A_{s,j}},\quad \psi_{s,j}|_{\partial B_1}=-b_s(t), \quad \psi_{s,j}|_{t=0}=\varphi_0,
\]
where $A_{s,j}=\int_{Q_T}\eta_j(u+\psi-s)FdVdt$, $\eta_j(x)$ is a smooth approximate of $x_+$ and satisfies $\eta(x)=\frac{1}{j}$ for $x\leq 0$. Since $u+\psi\leq 0$ on boundary $\partial B_1\times [0,T]$, we know that $\frac{\eta_j(u+\psi-s)F}{A_{s,j}}$ equals a positive constant on $\partial B_1\times [0,T]$. Therefore one can choose a function $b_s(t)$ satisfies the boundary compatibility condition and $b_s(t)\leq 1$. Let $\Omega_s:=\{u+\psi-s>0\}$.
Consider the function $\Psi=u+\psi-s-A(-\psi_{s,j})^{\frac{n+1}{n+2}},$ here $A$ is a positive constant to be specific later. We claim that $\Psi\leq 0$. 
On the parabolic boundary of $Q_T$, we have $\Psi\leq 0$. We can assume that the maximum is attained at $x_0\in \Omega_s$, then at $x_0$ we have \[
	\begin{aligned}
	0\geq &-\partial_t\Psi+a^{i\bar{j}}\partial_{i}\partial_{\bar{j}}\Psi \\
	\geq &-\partial_t u+a^{i\bar{j}}u_{i\bar{j}}+(-\partial_t\psi)+a^{i\bar{j}}\psi_{i\bar{j}}-\partial_t\psi_{s,t}+a^{i\bar{j}}A\frac{n+1}{n+2}(-\psi_{s,t})^{-\frac{1}{n+2}}(\psi_{s,j})_{i\bar{j}}\\
	\geq &-\partial_t u+a^{i\bar{j}}u_{i\bar{j}}+(-\partial_t\psi)+n\det(a^{i\bar{j}})^{\frac{1}{n}}\det(\psi_{i\bar{j}})^{\frac{1}{n}}\\
	&-\partial_t\psi_{s,t}+A\frac{n+1}{n+2}(-\psi_{s,t})^{-\frac{1}{n+2}}\det^{\frac{1}{n}}(a^{i\bar{j}})\det^{\frac{1}{n}}(\psi_{s,j})_{i\bar{j}}\\
	\geq &-\partial_t u+a^{i\bar{j}}u_{i\bar{j}}+(n+1)(-\partial_t\psi)^{\frac{1}{n+1}}\det(a^{i\bar{j}})^{\frac{1}{n+1}}\det(\psi_{i\bar{j}})^{\frac{1}{n+1}}\\
	&+(n+1)(-\partial_t\psi_{s,j})^{\frac{1}{n+1}}\det(a^{i\bar{j}})^{\frac{1}{n+1}}\det((\psi_{s,j})_{i\bar{j}})^{\frac{1}{n+1}}
	\end{aligned},
\]where we have used the arithmetic-geometric mean value inequality in the last line $$x+ny^{\frac{1}{n}}\geq (n+1)x^{\frac{1}{n+1}}y^{\frac{1}{n+1}}.$$
Let $e^G=\frac{(f^{-})^{n+1}}{(n+1)^{n+1}\mathscr{D}},$ (approximate $G$ from above by smooth functions equaling positive constant on the boundary $\partial B_1\times [0,T]$ if necessary),  then we get \[
	e^G\geq (-\partial_t\psi)\det(\psi_{i\bar{j}})+ \left(\frac{A(n+1)}{n+2}\right)^{n+1}(-\psi_{s,j})^{-\frac{n+1}{n+2}}(-\partial_t\psi_{s,j})\det((\psi_{s,j})_{i\bar{j}})
	\]
	Combine with \eqref{paraboliccomparison}, we get \[
         F\geq\left(\frac{A(n+1)}{n+2}\right)^{n+1}(-\psi_{s,j})^{-\frac{n+1}{n+2}}\eta_{j}(u+\psi-s)FA_{s,j}^{-1}.
	\]
	So if we choose $A=\left(\frac{n+2}{n+1}\right)^{\frac{n+1}{n+2}}A_{s,j}^{\frac{1}{n+2}},$ we get the desired estimate. 

Next we can begin the iteration process. Let $\phi(s)=\int_{\Omega_s}FdVdt$ and $A_s:=\int_{\Omega_s}(u+\psi-s)_+FdVdt.$ Then \[
	\begin{aligned}
	A_s\leq &c_nA_{s,j}^{\frac{1}{n+2}}\int_{\Omega_s}(-\psi_{s,j})^{\frac{n+1}{n+2}}FdVdt\\
	\leq & c_nA_{s}^{\frac{1}{n+2}}\left(\int_0^T\int_{B_1}(-\psi_{s,j})^{\frac{(n+1)p}{n+2}}FdVdt\right)^{\frac{1}{p}}\left(\int_{\Omega_s}FdVdt\right)^{1-\frac{1}{p}}
	\end{aligned}
\]
Choose $p>n+2$, then if we let $\Lambda:=\int_0^T\int_{B_1}(-\psi_{s,j})^{\frac{(n+1)p}{n+2}}FdVdt$, we get \[
	A_s\leq C(n,p)\Lambda^{\frac{n+2}{(n+1)p}}\phi(s)^{\frac{n+2}{n+1}(1-\frac{1}{p})}=C(n,p)\Lambda^{\frac{n+2}{(n+1)p}}\phi(s)^{1+\delta}.
\]
By De Giorgi iteration lemma, we get $\phi(s)=0$ for all $s\geq s_{\infty}$ with \[
	s_\infty=\frac{C(n,p)\Lambda^{\frac{n+2}{(n+1)p}}\phi(0)^\delta}{1-2^{-\delta}}.
\]
Since $F\leq e^G$, we have $\phi(0)\leq \int_0^T\int_{B_1}e^GdVdt$, and \[
	\Lambda\leq \min\{CT, \int_{Q_T}e^{q^*G}dVdt\}, \text{ where } \frac{1}{q^*}+\frac{1}{q}=1.
\]


\appendix 
\section{Existence of the inverse Monge-Ampere flow}
The existence of solution to equation \eqref{existenceofinvermongeampereflow} is guaranteed by the following a prior estimate\begin{theorem}
Let $u\in C^4{({Q_T}})\cap C^{2,1}(\overline{Q_T})$ be a smooth solution to the equation \eqref{existenceofinvermongeampereflow}, then there exists constant $C>0$ depending on $T,f,b,n$ such that \[
	|u|_{C^{2,1}(\overline{Q_T})}\leq C.
\]
\end{theorem}
\begin{proof}
We follow the ideas of \cite{MR0780073} and \cite{MR1664889}. Let $\underline{u}=-b(t)+A\varphi_0$, where $A$ is a constant chosen such that $b'(t)A^n\geq f$ for all $t\in[0,T].$ Therefore $\underline{u}$ satisfies $(-\partial_t\underline{u})(\ddc\underline{u}_{i\bar{j}})^n\geq f$ on $\overline{Q_T}$, which is a subsolution of equation \eqref{existenceofinvermongeampereflow}.
\begin{enumerate}
\item $ \underline{u}\leq u\leq -b(t)$.

Note that $\underline{u}\leq u$ on the parabolic boundary $\partial_P Q_T$, we can assume that the maximum of $\underline{u}-u$ is attained $x_0\in Q_T$. Then at $x_0$, we have $-\partial_t(\underline{u}-u)\leq 0$ and $\ddc(\underline{u}-u)\leq 0$. Therefore we have \[
	b'(t)A^n=(-\partial_t\underline{u})\det(\underline{u}_{i\bar{j}})\leq (-\partial_tu)\det(u_{i\bar{j}})=f.
\]We get a contradiction to our choice of $A$. Therefore the maximum is attained at the parabolic boundary of $Q_T$, and we get $\underline{u}\leq u$. For the other direction, just apply maximum principle to the $\PSH$-function $u$.

\item $M^{-1}\leq (-\partial_t u)\leq M$ for a positive constant $M>0$ depending on upper bound of $|(\log f)_t|, ,f, b'(t), T$.

let $\cL=-\partial_t+(-\partial_tu)u^{i\bar{j}}\partial_i\partial_{\bar{j}}$ be the linearization operator of the equation. Then taking derivative of the equation with respect to $t$, we get $\cL (-u_t)=(-\log f)_t(-u_t)$. Since $f$ is strictly positive and smooth, we can assume that $|(\log f)_t|\leq C_1$. Then \[
	\cL(-u_te^{-C_1t})=e^{-C_1t}\cL(-u_t)+C_1e^{-C_1t}(-u_t)=(C_1+(\log f)_t)e^{-C_1t}(-u_t)>0.
\]Hence by maximum principle, the maximum of $(-u_t)e^{-C_1t}$ is attained at the parabolic boundary of $Q_T$. On $\partial_1\times[0,T)$, we have $-u_t=b'(t)$ and on $B_1\times\{0\}$ we have $-u_t=\frac{f}{\det{(\varphi_0)_{i\bar{j}}}}$. Therefore $(-u_t)\leq Ce^{-C_1t}$. 

Similarly, consider the function $e^{C_1t}(-u_t)$, we have $\cL(e^{C_1t}(-u_t))<0$ and we get $(-u_t)\geq Ce^{-C_1t}$. 

\item $|\nabla u|\leq C.$ 

Note that for ant $t\in [0,T)$, we have $\underline{u}\leq u\leq -b(t)$ and $\underline{u}=u=-b(t)$ on $\partial B_1\times\{t\}$, we get $0\leq \partial_n u\leq \partial_n\underline{u}$ on $\partial B_1\times\{t\}$, where $\partial_n$ means the outer normal derivative. Therefore $|\nabla u|$ is bounded by a uniform constant $C$ on $\partial B_1\times\{t\}$. On the other hand, $|\nabla u|$ is bounded on the bottom boundary $B_1\times\{0\}$ since $u=\varphi_0$. 

For any unit vector $e$, we have \[
	\cL(\nabla_e u)=(\nabla_e\log f)(-u_t). 
\]
Note that $\sum_{i=1}^n u^{i\bar{i}}\geq n\det(u_{i\bar{j}})^{-\frac{1}{n}}=n(-u_t)^{\frac{1}{n}}f^{-\frac{1}{n}}$. Let $C_0$ be an upper bound of $f$ and $C_2$ be an upper bound of $|\nabla \log f|$. Then \[
	\cL(\pm \nabla_e u+C_4|z|^2)\leq -C_2M+C_4\sum_{i=1}^n u^{i\bar{i}}\geq-C_2M+C_4nM^{-\frac{1}{n}}C_0^{-\frac{1}{n}}>0, 
\] if we choose $C_4>C_2n^{-1}M^{1+\frac{1}{n}}C_0^{\frac{1}{n}}.$ By maximum principle we get \[
	\pm \nabla_e u\leq C_5(1+\sup_{\partial_PQ_T}|\nabla u|)\leq C.
\]

\item $\sup_{Q_T}|\nabla^2u|\leq C(1+\sup_{\partial_P Q_T}|\nabla^2 u|).$

Apply $\nabla_e$ to the equation we get \begin{equation}\label{fistorder}
	(-\partial_t\nabla_eu)\det(u_{i\bar{j}})+(-u_t)\det(u_{i\bar{j}})u^{k\bar{l}}(\nabla_eu)_{k\bar{l}}=\nabla_ef.
\end{equation}
Apply $\nabla_e$ again, we get \begin{equation}\label{secondorder}\begin{aligned}
	&(-\partial_t\nabla_{ee}^2u)\det(u_{i\bar{j}})+(-\partial_t\nabla_eu)\det(u_{i\bar{j}})u^{k\bar{l}}(\nabla_eu)_{k\bar{l}}\\
	&+(-\partial_t\nabla_eu)\det(u_{i\bar{j}})u^{k\bar{l}}(\nabla_eu)_{k\bar{l}}+(-u_t)\det(u_{i\bar{j}})u^{m\bar{n}}(\nabla_eu)_{m\bar{n}}u^{k\bar{l}}(\nabla_eu)_{k\bar{l}}\\
	&-(-u_t)\det(u_{i\bar{j}})u^{k\bar{q}}u^{p\bar{l}}(\nabla_eu)_{p\bar{q}}(\nabla_eu)_{k\bar{l}}+(-u_t)\det(u_{i\bar{j}})u^{k\bar{l}}(\nabla_{ee}^2u)_{k\bar{l}}=\nabla_{ee}^2f
	\end{aligned}
\end{equation}
Plug in the equation \eqref{fistorder} to \eqref{secondorder}, and divided it by $\det(u_{i\bar{j}})$, we have \[\begin{aligned}
	\cL(\nabla_{ee}^2u)-(-u_t)\det(u_{i\bar{j}})u^{m\bar{n}}(\nabla_eu)_{m\bar{n}}u^{k\bar{l}}(\nabla_eu)_{k\bar{l}}&-(-u_t)\det(u_{i\bar{j}})u^{k\bar{q}}u^{p\bar{l}}(\nabla_eu)_{p\bar{q}}(\nabla_eu)_{k\bar{l}}\\
&=\frac{(-u_t)\nabla_{ee}^2f}{f}-2\frac{(-u_t)\nabla_ef}{f}
	\end{aligned}
\]Hence, \[
	\cL(\nabla_{ee}u^2)\geq -C_6, \quad \text{for some constant $C_6$ depending on } \sup f^{-1}, ||f||_{C^2(Q_T)}. 
\]
Let $C_7=C_6n^{-1}M^{\frac{1}{n}}C_0^{\frac{1}{n}}+1$, we have \[
	\cL(\nabla_{ee}^2u+C_7|z|^2)\geq-C_6+C_7\sum_{i=1}^{n}u^{i\bar{i}}\geq 1.
\]
Apply maximum principle to $\nabla^2_{ee}u+C_7|z|^2$ we get the desired estimate. 
\item $\sup_{\partial B_1\times\{t\}}|\nabla_{ee}^2u|\leq C$ for any unit vector tangent to $\partial B_1\times\{t\}$.

For any point $p\in \partial B_1$, we can choose coordinate $(z_1,\ldots,z_n)$ such that $p=(0,\ldots,0)$.  We use $z_i=s_{2i-1}+\sqrt{-1}s_{2i}$ to denote the real coordinates. One can assume that $s_{2n}$ is the outer normal direction of $\partial B_1$ at $p$. Denote $s'=(s_1,\ldots,s_{2n-1})$. Then for a neighborhood of $p$, say $U$, one has \[
	U\cap B_1=\{s_{2n}>\rho(s')=\frac{1}{2}\sum_{\alpha,\beta\in\{1,\ldots,2n-1\}}B_{\alpha\bar{\beta}}s_\alpha s_{\beta}+O(|s'|^3)\}.
\]
Since $u=-b(t)$ on $\partial B_1$, i.e. $u(s',\rho(s'))=-b(t)$, we have \[
	u_{s_\alpha s_{\beta}}(0)+u_{s_{2n}}(0)B_{\alpha\beta}=0.
\]
It follows \[|\nabla_{ee}^2u|\leq C\sup|\nabla u|, \quad \text{ for any tangent vector $e$.}\]

\item $\sup_{\partial B_1\times\{t\}}|\nabla^{2}_{\nu e}u|\leq C$ for any unit tangent vector $e$ and outer normal vector $\nu$. 

To estimate $u_{s_\alpha s_{2n}}$, we consider the operator $T:=\partial_{s_\alpha}+\rho_{s_{\alpha}}\partial_{s_{2n}}$ in a neighborhood of $p\in \partial B_1\times\{t\}$.  Then direct computation shows\[
	\begin{aligned}
	\cL(T(u-\underline{u}))=&-\partial_t\partial_{s_{\alpha}}u-\rho_{s_\alpha}\partial_tu_{s_{2n}}+(-u_t)u^{i\bar{j}}(\partial_{s_\alpha}u)_{i\bar{j}}+(-u_t)u^{i\bar{j}}(\rho_{s_{\alpha}}u_{x_{s_{2n}}i\bar{j}})\\
	&+(-u_t)u^{i\bar{j}}(\rho_{s_\alpha i}u_{s_{2n}\bar{j}}+\rho_{s_\alpha\bar{j}}u_{s_{2n}i}+\rho_{s_\alpha i\bar{j}}u_{s_{2n}})\\
	=& (-u_t)T(\log f)+(-u_t)u^{i\bar{j}}(\rho_{s_\alpha i}u_{s_{2n}\bar{j}}+\rho_{s_\alpha\bar{j}}u_{s_{2n}i}+\rho_{s_\alpha i\bar{j}}u_{s_{2n}})
	\end{aligned}
\]
Note that $\partial_{s_{2n}}=2\sqrt{-1}\partial_{z_n}-\sqrt{-1}\partial_{s_{2n-1}}$, we get \[\begin{aligned}
	(-u_t)u^{i\bar{j}}\rho_{s_\alpha i}(2\sqrt{-1}u_{z_n\bar{j}}-\sqrt{-1}u_{s_{2n-1}\bar{j}})=&2\sqrt{-1}(-u_t)\rho_{s_{\alpha}n}-\sqrt{-1}(-u_t)\rho_{s_\alpha i}u_{s_{2n-1}\bar{j}}\\
	=&O(1+(\sum u^{k\bar{k}})^{\frac{1}{2}}(u_{s_{2n-1}\bar{j}}u_{s_{2n-1}i}u^{i\bar{j}})^{\frac{1}{2}}).
	\end{aligned}
\]
Similar computation for the other two terms $\rho_{s_\alpha\bar{j}}u_{s_{2n}i}$, $\rho_{s_\alpha i\bar{j}}u_{s_{2n}}$ shows that \begin{equation}\label{tnestimatemainterm}
	\cL(T(u-\underline{u}))=(-u_t)T\log f+O(1+(\sum u^{k\bar{k}})^{\frac{1}{2}}(u_{s_{2n-1}\bar{j}}u_{s_{2n-1}i}u^{i\bar{j}})^{\frac{1}{2}}).
\end{equation}
To kill the last term, we consider the function $(u_{s_{2n-1}}-\underline{u}_{s_{2n-1}})^2$ defined in a neighborhood of $p$. \begin{equation}\label{u2nsquare}
\begin{aligned}
\cL(u_{s_{2n-1}}-\underline{u}_{s_{2n-1}})^2=&2(u_{s_{2n-1}}-\underline{u}_{s_{2n-1}})(-\partial_t)(u_{s_{2n-1}}-\underline{u}_{s_{2n-1}})\\
&+2(-u_t)(u_{s_{2n-1}}-\underline{u}_{s_{2n-1}})u^{i\bar{j}}(u_{s_{2n-1}}-
\underline{u}_{s_{2n-1}})_{i\bar{j}}\\
&+2(-u_t)u^{i\bar{j}}(u_{s_{2n-1}}-\underline{u}_{s_{2n-1}})_i(u_{s_{2n-1}}-\underline{u}_{s_{2n-1}})_{\bar{j}}\\
\geq &2(u_{s_{2n-1}}-\underline{u}_{s_{2n-1}})\cL((u_{s_{2n-1}}-\underline{u}_{s_{2n-1}}))\\
&+(-u_t)(u_{s_{2n-1}\bar{j}}u_{s_{2n-1}i}u^{i\bar{j}})-|\nabla^2 \underline{u}|_{u_{i\bar{j}}}
\end{aligned}
\end{equation}
Combine \eqref{tnestimatemainterm} and \eqref{u2nsquare}, there exists a constant $C_8,C_9$ such that \[
	\cL(\pm T(u-\underline{u})-C_8(u_{s_{2n-1}}-\underline{u}_{s_{2n-1}})^2)\leq C_9
\]

Now we define a barrier function $v$ as \[
	v=(u-\underline{u})+a(-b(t)-\underline{u})-Nd^2,
\] where $a,N$ are constant to be determined and $d$ is the distance function $d(z)=dist(z,\partial B_1),$ which is defined in a neighborhood of $p$. 
We compute $\cL v$ as follows\[
	\cL(u-\underline{u})=-u_t+(-u_t)u^{i\bar{j}}u_{i\bar{j}}+\underline{u}_t-(-u_t)u^{i\bar{j}}\underline{u}_{i\bar{j}}=(n+1)(-u_t)+b'(t)-(-u_t)\sum_{i=1}^{n}u^{i\bar{i}}.
\]
\[
	\cL(-b(t)-\underline{u})=b'(t)+\underline{u}_t-(-u_t)u^{i\bar{j}}\underline{u}_{i\bar{j}}=-(-u_t)\sum_{i=1}^{n}u^{i\bar{i}}.
\]
And \[
	\cL d^2=2du^{i\bar{j}}d_{i\bar{j}}+2u^{i\bar{j}}d_id_{\bar{j}}=u^{n\bar{n}}+O(d)\sum_{i}^{n}u^{i\bar{i}}.
\]
Note that  \[
	\frac{1}{4}\sum_{i=1}^nu^{i\bar{i}}+Nu^{n\bar{n}}\geq \frac{n}{4}(N\Pi u^{i\bar{i}})^{-\frac{1}{n}}=\frac{nN^{\frac{1}{n}}(-u_t)^{\frac{1}{n}}}{4f^{\frac{1}{n}}}\geq C_{10}N^{\frac{1}{n}}.
\]
So consider the above functions in the ball $B(p,\delta)$, we have \[
	\cL v\leq n(-u_t)+b'(t)+(-1-a+N\delta+\frac{1}{4})\sum_{i=1}^{n}u^{i\bar{i}}-C_{10}N^{\frac{1}{n}}.
\]
It shows that $\cL v$ can be made arbitrary small if we choose $N$ large and $\delta$ small. Hence \begin{equation}\label{differentialinequalityTN}
	\cL(v+C_{11}|z|^2 \pm T(u-\underline{u})-C_8(u_{s_{2n-1}}-\underline{u}_{s_{2n-1}})^2))<0
\end{equation}

On the one hand, from our construction of $\underline{u}$, we know $(-b(t)-\underline{u})\geq C_{12}d$ for a uniform positive constant $C_{12}$. Therefore $a(-b(t)-\underline{u})-Nd^2\geq 0$ for arbitrary $a,N$ by shrinking $\delta$ if necessary. So $v\geq 0$ on $B(p,\delta)\times [0,t].$
On the other hand, $T(u-\underline{u})=0$ on $(\partial B_1\times[0,t]) \cap (B(p,\delta)\times[0,t])$, and $|T(u-\underline{u})|\leq C_{13}$ on $(B_1\times[0,t])\cap(\partial B(p,\delta)\times[0,t])$ thanks to the gradient estimate of $u$. For the term $(u-\underline{u})_{s_{2n-1}}$, we have \[
	(u-\underline{u})_{s_{2n-1}}+(u-\underline{u})_{s_{2n}}\rho_{s_{2n-1}}=0
\text{ on } \partial B_1\times [0,T].\]
Therefore, \[
	(u_{s_{2n-1}}-\underline{u}_{s_{2n-1}})^2\leq C_{14}|z|^2
\text{ on } \partial B_1\times[0,T].\]
On the other part of boundary $(\partial B(p,\delta)\times[0,t])\cap (B_1\times[0,t])$, we have $|z|^2\geq \delta^2$. In summary we have \begin{equation}\label{boundarypositiveTN}
	v+C_{11}|z|^2 \pm T(u-\underline{u})-C_8(u_{s_{2n-1}}-\underline{u}_{s_{2n-1}})^2)\geq 0 \text{ on } \partial_P((B(p,\delta)\cap B_1)\times[0,t]).
\end{equation}For sufficient large $C_{11}$ and small $\delta$.
With \eqref{differentialinequalityTN}, \eqref{boundarypositiveTN} and maximum principle, we get \[
	v+C_{11}|z|^2 \pm T(u-\underline{u})-C_8(u_{s_{2n-1}}-\underline{u}_{s_{2n-1}})^2)\geq 0
	\] and therefore
	\[
	|u_{s_{2n}s_{2n}}|\leq v_{s_{2n}}(0)+|\underline{u}_{s_\alpha s_{2n}}|\leq C. 
	\]

\item $\sup_{\partial B_1\times\{t\}}|\nabla^{2}_{\nu\nu}u|\leq C$.

From $\det(u_{i\bar{j}})=\frac{f}{(-u_t)}\leq C$, to show an upper bound of $u_{s_{2n}s_{2n}}$, one needs only to a positive lower bound of $u_{s_{\alpha}s_{\beta}}$. Near the boundary $\partial B_1$, we can write the function $u$ as \[
	u=-b(t)+\sigma\varphi_0,
\] for a smooth function $\sigma$ near $\partial B_1$. Hence, \[
	u(p)_{s_\alpha s_\beta}=-u_{s_{2n}}(0)\varphi_{s_\alpha s_\beta}.
\]Now observe that \[
	\Delta u\geq n\det(u_{i\bar{j}})^{\frac{1}{n}}\geq C_{15}>0.
\] By Hopf lemma, there exists a constant $C_{16}>0$ such that $u_{s_{2n}}\leq -C_{16}$, which implies the desired estimate. Here we use the fact $u=$constant on the boundary $\partial B_1$, which simplifies the proof a lot.

\end{enumerate}
\end{proof}

\AtNextBibliography{\small}
\begingroup
\setlength\bibitemsep{2pt}
\printbibliography

\end{document}